\documentclass[11pt,twoside]{elsarticle}
\usepackage{amssymb,amsmath,amsfonts}
\usepackage{latexsym}

\usepackage[english]{babel}

\newtheorem{theorem}{Theorem}[section]
\newtheorem{lemma}[theorem]{Lemma}

\newtheorem{prop}[theorem]{Proposition}
\newtheorem{cor}[theorem]{Corollary}

\newtheorem{problem}[theorem]{Problem}
\newtheorem{example}[theorem]{Example}
\newproof{proof}{Proof}

\begin{document}
\title{On the sum of a narrow and a compact operators}

\author{Volodymyr Mykhaylyuk}

\address{Department of Applied Mathematics\\
Chernivtsi National University\\
str.~Kotsyubyns'koho 2, Chernivtsi, 58012 (Ukraine)}

\ead{vmykhaylyuk@ukr.net}


\begin{abstract}
Our main technical tool is a principally new property of compact narrow operators which works for a domain space without an absolutely continuous norm. It is proved that for every K\"{o}the $F$-space $X$ and for every locally convex $F$-space $Y$ the sum $T_1+T_2$ of a narrow operator $T_1:X\to Y$ and a compact narrow operator $T_2:X\to Y$ is a narrow operator. This gives a positive answers to questions asked by M.~Popov and B.~Randrianantoanina (\cite[Problem 5.6 and Problem 11.63]{PR})
\end{abstract}

\maketitle

\section{Introduction}

The notion of narrow operators was introduced by A.~Plichko and M.~Popov in \cite{PP}. This notion was considered as some development of the notion of a compact operator. Narrow operators were studied by many mathematicians (see \cite{PR}). The question of whether a sum of two narrow operators has to be narrow cause a special interest in the investigation of properties of narrow operators. The study of this question was conducted in two directions.  The first of them consists of results on the narrowness of a sum of two narrow operators. It was obtained in \cite{PP} that a sum of two narrow operators on $L_1$ is narrow. The question on the narrowness of a sum of two narrow operators defined on $L_1$ with any range Banach space was arised in this connection (see  \cite{KP1}). In \cite{MMP} the notion of narrow operators defined on K\"{o}the function spaces was  extended to operators defined on vector lattices. Results obtained in \cite{MMP} imply that a sum $T_1+T_2$ of two regular narrow operators $T_1,T_2:X\to Y$ defined on an  order continuous atomless Banach lattice $X$ and with values in an  order continuous Banach lattice $Y$ is narrow.

The second direction of the investigations contains examples of two narrow operators the sum of which is non-narrow.  It was proved in \cite{PP} that
for any rearrangement invariant space $X$ on $[0,1]$ with an unconditional basis every operator on $X$ is a sum of two narrow operators. In \cite{MP}
 M.Popov and the author proved that for any K\"{o}the Banach space $X$ on $[0,1]$, there exist a Banach space $Y$ and narrow operators $T_1,T_2 :X \to Y$ with a non-narrow sum $T = T_1 + T_2$. This answers in the negative the question of V.~Kadets and M.~Popov from \cite{KP1}. Moreover, an example of regular narrow operators $T_1, T_2:L_p\to L_\infty$, where $p\in (1,\infty]$, with a non-narrow sum $T_1+T_2$ was constructed in \cite{MP}.

Furthermore, the narrowness of the sum of a narrow and a compact operators was investigated too. It was proved in \cite{PP} that the sum of a narrow and a compact operators defined on a symmetric Banach space on $[0,1]$ with an absolutely continuous norm is narrow. Since $L_\infty$ is the classical example of a symmetric Banach space the norm of which is not absolutely continuous, the following questions are natural.

\begin{problem} \label{prob:1}(\cite[Problem~5.6]{PR})
Is a sum of two narrow functionals on $L_\infty$ narrow?
\end{problem}

\begin{problem} \label{prob:2}(\cite[Problem~11.63]{PR})
Is a sum of two narrow operators on $L_\infty$, at least one of which is compact, narrow?
\end{problem}

It worth mentioning that a compact operator defined on $L_\infty$ need not be narrow. Moreover, there is a nonnarrow continuous linear functional \cite{MMP}. Thus, instead of asking of whether the sum of a narrow and a compact operators is narrow, one should ask of whether the sum of a narrow and a compact narrow operators is narrow.

In this paper we obtain a principally new property of compact narrow operator which works for a domain space without an absolutely continuous norm. Using this property we give a positive answer to problems \ref{prob:1} and \ref{prob:2}. More precisely, we show that for every K\"{o}the $F$-space $X$ and for every locally convex $F$-space $Y$ the sum of a narrow operator $T_1:X\to Y$ and a compact narrow operator $T_2:X\to Y$ is narrow.

\section{Preliminaries}

For topological vector spaces $X$ and $Y$ by $\mathcal L(X,Y)$ we denote the space of all linear continuous operators $T:X\to Y$.

Let $(\Omega, \Sigma, \mu)$ be a finite atomless measure space, let $\Sigma^+$ be the set of all $A\in\Sigma$ with $\mu(A)>0$ and let   $L_0(\mu)$ be the linear space of all equivalence classes of $\Sigma$-measurable functions $x: \Omega \to \mathbb K$, where $\mathbb K \in \{\mathbb R, \mathbb C\}$.

By $\mathbf{1}_A$ we denote the characteristic function of a set $A \in \Sigma$. The notation $A = B \sqcup C$ means that $A = B \cup C$ and $B \cap C = 0$. By a \textit{sign} we mean any $\{-1,0,1\}$-valued element $x \in L_0(\mu)$. A sign $x$ is called a \textit{sign on a set} $A\in \Sigma$ provided ${\rm supp} \, x = A$. A sign $x \in L_0(\mu)$ is said to be \textit{of mean zero} if $\int xd\mu=0$.

 A complete metric linear space with an invariant metric $\rho(x,y) = \rho(x+z,y+z)$ and the corresponding $F$-norm $\|x\| = \rho(x,0)$ is called an {\it $F$-space}. With no loss of generality we may additionally assume that for every $x\in X$ and every scalar $\alpha$ with $|\alpha|\leq 1$ one has $\|\alpha x\|\leq \|x\|$. An $F$-space $X$ which is a linear subspace of $L_0(\mu)$ is called a \textit{K\"{o}the $F$-space on} $(\Omega, \Sigma, \mu)$ if $\mathbf{1}_\Omega \in X$, and for each $x \in L_0(\mu)$ and $y \in X$ the condition $|x| \leq |y|$ implies $x \in X$ and $\|x\| \leq \|y\|$. If, moreover, $X$ is a Banach space and $X \subseteq L_1(\mu)$ then $X$ is called a \textit{K\"{o}the Banach space on} $(\Omega, \Sigma, \mu)$.

A K\"{o}the $F$-space $X$ on $(\Omega, \Sigma, \mu)$ is said to have an {\it absolutely continuous norm}, if $\lim_{\mu(A) \to 0} \bigl\|x \cdot \mathbf{1}_{A} \bigr\| = 0$ for every $x \in X$, and an \textit{absolutely continuous norm on the unit} if $\lim_{\mu(A) \to 0} \|\mathbf{1}_A\| = 0$.

Let $X$ be a K\"{o}the $F$-space on $(\Omega, \Sigma, \mu)$, and let $Y$ be an $F$-space. An operator $T \in \mathcal L(X,Y)$ is called \textit{narrow} if for every $A \in \Sigma^+$ and $\varepsilon > 0$ there is a mean zero sign $x$ on $A$ with $\|Tx\| < \varepsilon$. An operator $T \in \mathcal L(X,Y)$ is called \textit{strictly narrow} if for every $A \in \Sigma^+$ there is a mean zero sign $x$ on $A$ with $Tx =0$.

Let $\alpha$ be an at most countable ordinal, $(X,\rho)$ be a metric linear space and $(x_\xi:\xi<\alpha)$ be a family of $x_\xi\in X$. Using induction in $\alpha$, we define the sum $\sum\limits_{\xi<\alpha}x_\xi$. If $\alpha=\beta+1$ then $\sum\limits_{\xi<\alpha}x_\xi=\sum\limits_{\xi<\beta}x_\xi+x_\beta$, and if $\alpha$ is a limited ordinal then $\sum\limits_{\xi<\alpha}x_\xi=\lim\limits_{\beta\to\alpha}\sum\limits_{\xi<\beta}x_\xi$. One can analogously introduce the sum $\sum_{i\in I}x_i$ where $I$ is an at most countable well-ordered set.

Let $X$ be an $F$-space and $I$ be a well-ordered set. A family $(x_i:i\in I)$ of $x_i\in X$ is called {\it a transfinite basis of $X$}(see \cite[p.581]{S}) if for every $x\in X$ there exists a unique family $(a_i:i\in I)$ of scalars $a_i\in I$ such that the set $J=\{i\in I:a_i\ne 0\}$ is at most countable and $x=\sum_{i\in J}a_i x_i=\sum_{i\in I}a_i x_i$. Moreover, the family of linear functionals $f_i:X\to\mathbb K$, $f_i(x)=a_i$, where $x=\sum_{j\in I}a_j x_j$, is called {\it the associated family of coefficient functionals}.

\section{A property of compact narrow operators}

\begin{prop} \label{pr:3}
Let $X$ be a K\"{o}the $F$-space on $(\Omega, \Sigma, \mu)$, $Z$ be a set of all signs in $X$, $Y$ be an $F$-space, $T\in\mathcal L(X,Y)$ be a narrow operator such that the set $K=\overline{T(Z)}$ is a compact subset of $Y$ and $\varepsilon>0$. Then there exists a partition $\Omega=\bigsqcup\limits_{k=1}^n A_n$ of $\Omega$ into measurable sets $A_k\in \Sigma$ such that $\|T x\|\leq \varepsilon$ for every $k\leq n$ and every sign $x$ with ${\rm supp}\,x\subseteq A_k$.
\end{prop}

\begin{proof} Assume on the contrary that the hypothesis is false. Using induction in $n\in\mathbb N$ we show that for each $n\in\mathbb N$ there exists a collection $(x_k)_{k=1}^{n}$ of signs $x_k$ on mutually disjoint sets $B_k$ such that $\|T x_k\|\geq \frac{\varepsilon}{2}$ for every $k\leq n$. For $n=1$ the statement is obvious. Let it be true for $n=k$. Prove it for $n=k+1$.

Let $(x_i)_{i=1}^{k}$ be a collection of signs $x_i$ on mutually disjoint sets $B_i$ such that $\|T x_i\|\geq \frac{\varepsilon}{2}$ for every $i\leq k$. Set $A=\Omega\setminus \bigsqcup\limits_{i=1}^k B_i$. If there exist a measurable set $B_{k+1}\subseteq A$ and a sign $x_{k+1}$ on $B_{k+1}$ such that $\|Tx_{k+1}\|\geq \frac{\varepsilon}{2}$ then the collection $(x_i)_{i=1}^{k+1}$ is as desired. Otherwise, by the assumption there exist $i\leq k$, a measurable set $B\subseteq B_i$ and a sign $z$ on $B$ such that $\|Tz\| > \varepsilon$. Say, let $i=k$. Using the narrowness of $T$ we decompose the set $B$ into measurable sets $B'_{k}$ and $B'_{k+1}$ such that $\|T z_k\|\geq \frac{\varepsilon}{2}$ and $\|T z_{k+1}\|\geq \frac{\varepsilon}{2}$, where $z_k=x_k\cdot \mathbf{1}_{B_k}$ and $z_{k+1}=x_k\cdot \mathbf{1}_{B_{k+1}}$. Then the collection $x_1, \dots, x_{k-1}, z_k, z_{k+1}$ is as desired.

Set  $V=\{y\in Y: \|y\|<\frac{\varepsilon}{4}\}$. Using the compactness of $K$ we choose $n\in\mathbb N$ and a circled neighborhood of zero $V_1$ in $Y$ such that $K+V_1\subseteq nV$. We take a 0-neighborhood $W\subseteq V$ in $Y$ and a collection $(y_k)_{k=1}^{m}$ of points $y_k\in K$ such that $\mathop{\underbrace{W+\dots +W}}\limits_{n}\subseteq V_1$ and $K\subseteq \bigcup\limits_{k=1}^m (y_k+W)$. Denote $N=mn$. Using the claim proved above we find a collection $(x_i)_{i=1}^{N}$ of signs $x_i$ on mutually disjoint sets $B_i$ such that $\|T x_i\|\geq \frac{\varepsilon}{2}$ for every $i\leq N$. Since $\{Tx_i:1\leq i\leq mn\}\subseteq \bigcup\limits_{k=1}^m (y_k+W)$, there exists an integer $k\leq m$ such that $|\{i\leq N: Tx_i\in y_k+W\}|\geq n$. Denote $y_0=y_k$. Without loss of generality we may and do assume that $Tx_i\in y_0+W$ for every $i\leq n$. Since $\|T x_1\|\geq \frac{\varepsilon}{2}$ and $W\subseteq V=\{y\in Y: \|y\|<\frac{\varepsilon}{4}\}$, $\|y_0\|\geq \frac{\varepsilon}{4}$, that is $y_0\not\in V$.

Let $z=\sum\limits_{i=1}^n x_i$. Since the sets $B_i$ are mutually disjoint, $z\in Z$ and $y=Tz=\sum\limits_{i=1}^n Tx_i\in K$. Moreover, $$y\in ny_0+\mathop{\underbrace{W+\dots +W}}\limits_{n}\subseteq ny_0+V_1.$$
Now we have $ny_0\in K-V_1=K+V_1\subseteq nV$, that is $y_0\in V$, which implies a contradiction.
$\Box$\end{proof}

\section{The sum of a narrow and a finite rank operators}

\begin{prop} \label{pr:1}
Let $X$ be a K\"{o}the $F$-space on $(\Omega, \Sigma, \mu)$, $Z$ be the set of all signs in $X$, $Y_1, Y_2$ be $F$-spaces, $T_1\in\mathcal L(X,Y_1)$ be a narrow operator, $T_2\in\mathcal L(X,Y_2)$ be an operator such that the set $K=T_2(Z)$ is relatively compact in $Y_2$, $\sigma>0$, $0<\gamma<\varepsilon$ and $\delta>0$ such that  $\|T_2\mathbf{1}_A\|\leq \frac{\gamma}{2}$ if $\mu(A)\leq \delta$. Then there exists a mean zero sign $x$ on $\Omega$ such that $\|T_1x\|\leq \sigma$ and $\|T_2x\|\leq \varepsilon$.
\end{prop}

\begin{proof}
Set $\varepsilon_1=\varepsilon-\gamma$. Using the narrowness of $T_1$ we construct a sequence of probability independent mean zero signs $r_n^{(1)}$ on $A_1=\Omega$ such that $\|T_1 r_n^{(1)}\|\leq \frac{\sigma}{2^2}$ for every $n\in\mathbb N$. Since $K$ is relatively compact, there exist distinct $m_1,n_1\in\mathbb N$ such that for $x_1=\frac{1}{2}(r_{n_1}^{(1)}-r_{m_1}^{(1)})$ we have $\|T_2x_1\|<\frac{\varepsilon_1}{2}$. Set $B_1={\rm supp}\, x_1$ and $A_2=A_1\setminus B_1$. It is clear that $\mu(A_2)=\mu(B_1)=\frac{1}{2}\mu(\Omega)$, $x_1$ is a mean zero sign on $B_1$ and $\|T_1x_1\|\leq \frac{\sigma}{2}$. Now we choose a sequence of probabilistic independent mean zero signs $r_n^{(2)}$ on $A_2$ such that $\|T_1 r_n^{(2)}\|\leq \frac{\sigma}{2^3}$ for every $n\in\mathbb N$, and find $n_2,m_2\in\mathbb N$ so that for $x_2=\frac{1}{2}(r_{n_2}^{(2)}-r_{m_2}^{(2)})$ we have $\|T_2x_2\|<\frac{\varepsilon_1}{2^2}$. Set $B_2={\rm supp}\, x_2$ and $A_3=A_2\setminus B_2$. It is clear that $\mu(A_2)=\mu(B_1)=\frac{1}{2^2}\mu(\Omega)$, $x_2$ is a mean zero sign on $B_2$ and $\|T_1x_2\|\leq \frac{\sigma}{2^2}$.

Repeating this process we obtain a sequence $(B_n)_{n=1}^{\infty}$ of measurable sets $B_n$ and a sequence $(x_n)_{n=1}^{\infty}$ of mean zero signs $x_n$ on $B_n$ which satisfy the following conditions:

$(a)$\,\,\, $B_n\cap B_m=\emptyset$ for arbitrary distinct $n,m\in\mathbb N$;

$(b)$\,\,\, $\mu(B_n)=\frac{\mu(\Omega)}{2^{n}}$ for every $n\in\mathbb N$;

$(c)$\,\,\, $\|T_1x_n\|\leq \frac{\sigma}{2^{n}}$ and $\|T_2x_n\|\leq \frac{\varepsilon_1}{2^{n}}$ for every $n\in\mathbb N$.

Using the proposition assumption, choose $m\in\mathbb N$ so that $\frac{\mu(\Omega)}{2^m}\leq \delta$. Set $B=\bigcup\limits_{n=1}^m B_n$ and
choose a mean zero sign $z$ on $A=\Omega\setminus B$ such that $\|T_1z\|\leq \frac{\sigma}{2^m}$. Note that $\|T_2 z\|\leq \gamma$ and $x=\sum\limits_{n=1}^m x_n +z$ is a mean zero sign on $\Omega$. Now we have
$$
\|T_1x\|\leq \sum\limits_{n=1}^m\|T_1 x_n\| +\|T_1 z\|\leq
\sum\limits_{n=1}^m\frac{\sigma}{2^n}+ \frac{\sigma}{2^m}=\sigma
$$
and
$$
\|T_2x\|\leq \sum\limits_{n=1}^m\|T_2 x_n\|+ \|T_2 z\| \leq
\sum\limits_{n=1}^m\frac{\varepsilon_1}{2^n}+\gamma\leq\varepsilon_1+\gamma=\varepsilon.
$$

$\Box$\end{proof}

\begin{cor} \label{cor:1}
Let $X$ be a K\"{o}the $F$-space on $(\Omega, \Sigma, \mu)$, $Z$ be the set of all signs in $X$, $Y$ be an $F$-space, $T_1\in\mathcal L(X,Y)$ be a narrow operator, $T_2\in\mathcal L(X,Y)$ be an operator such that the set $T_2(Z)$ is relatively compact in $Y$ and $\lim\limits_{\mu(A)\to 0}\|T_2\mathbf{1}_A\|=0$. Then  $T=T_1+T_2$ is a narrow operator.
\end{cor}

We need the next well-known lemma (see \cite[lemma 2.1.2]{KK}, \cite[lemma 10.20]{PR}).

\begin{lemma} \label{l:1}
Let $(x_i)_{i=1}^n$ be a finite sequence of vectors in a finite dimensional normed space $X$ and $(\lambda_i)_{i=1}^n$ be reals with $0\leq \lambda_i\leq n$ for each $i$. Then there exists a sequence $(\theta_i)_{i=1}^n$ of numbers $\theta_i\in\{0,1\}$ such that $$\|\sum\limits_{k=1}^{n}(\lambda_i-\theta_i)x_i\|\leq \frac{{\rm dim}X}{2}\max\limits_i\|x_i\|.$$
\end{lemma}

\begin{theorem} \label{th:1}
Let $X$ be a K\"{o}the $F$-space on $(\Omega, \Sigma, \mu)$, $Y_1, Y_2$ be $F$-spaces, $T_1\in\mathcal L(X,Y_1)$ be a narrow operator, $T_2\in\mathcal L(X,Y_2)$ be a finite rank narrow operator and $\sigma, \varepsilon>0$. Then there exists an of mean zero sign $x$ on $\Omega$ such that $\|T_1x\|\leq \sigma$ and $\|T_2x\|\leq \varepsilon$.
\end{theorem}

\begin{proof}
 Set $m={\rm dim}T_2(X)$ and choose $\delta>0$ such that for each $y\in T_2(X)$ the inequality $p(y)\leq \delta$ implies the inequality $\|y\|\leq \varepsilon$, where $p$ is a fixed norm on $T_2(X)$. Using Proposition \ref{pr:3} we choose a partition $\Omega=\bigsqcup\limits_{k=1}^n A_n$ of $\Omega$ into measurable sets $A_k\in \Sigma$ such that $p(T_2 x)\leq \frac{\delta}{2m}$ for every $k\leq n$ and every sign $x$ with ${\rm supp}\,x\subseteq A_k$. Using the narrowness of $T_1$ for every $k\leq n$ we choose a mean zero sign $x_k$ on $A_k$ such that $\|T_1x_k\| \leq \frac{\sigma}{2^{k}}$. Note that $p(T_2 x_k)\leq \frac{\delta}{m}$. Using Lemma \ref{l:1} we choose a collection $(\theta_k)_{k=1}^{n}$, $\theta_k\in \{-1,1\}$ such that $p(\sum\limits_{k=1}^{n}\theta_k T_2 x_k)\leq \delta$. Set $x=\sum\limits_{k=1}^{n}\theta_k \,x_k$. It is clear that $x$ is a mean zero sign on $\Omega$. By the choice $\delta$ we have $\|T_2 x\|\leq \varepsilon$. Moreover, $$\|T_1 x\|\leq\sum\limits_{k=1}^{n}\|T_1 x_k\| \leq\sum\limits_{k=1}^{n} \frac{\varepsilon}{2^{k}}\leq \sigma.$$
$\Box$\end{proof}

\begin{cor} \label{cor:2}
Let $X$ be a K\"{o}the $F$-space on $(\Omega, \Sigma, \mu)$, $Z$ be the set of all signs in $X$, $Y$ be an $F$-space, $T_1\in\mathcal L(X,Y)$ be a narrow operator, $T_2\in\mathcal L(X,Y)$ be an operator, the restriction of which $T_2|_{Z}$ is a uniform limit of a sequence $(S_n|_{Z})_{n=1}^\infty$ where all operators $S_n\in\mathcal L(X,Y)$ are narrow and finite rank. Then the operator $T=T_1+T_2$ is narrow.
\end{cor}

\section{The sum of a narrow and a compact operators}

\begin{prop} \label{pr:4}
Let $I$ be a well-ordered set, $X$ be an $F$-space with a transfinite basis $(e_i:i\in I)$. Then the coefficient functionals are continuous and the family $(f_i:i\in I)$ is equicontinuous if $\|e_i\|\geq 1$ for every $i\in I$.

If, moreover, $K$ is a compact subset of $X$ and $\varepsilon>0$, then there exists a finite set $J\subseteq I$ such that $\|x-\sum_{i\in J}f_i(x)e_i\|\leq \varepsilon$ for every $x\in K$.
\end{prop}

\begin{proof} We argue analogously as in the case of a separable Banach space $X$ (see \cite[p.115]{Sh1}). Let $I_0=I\sqcup \{i_0\}$ and $i<i_0$ for every $i\in I$. Consider the function $$p(x)=\sup_{i\in I_0}\|\sum_{j< i}f_i(x)e_i\|,$$ which is an $F$-norm on $X$. Note that $\|x\|\leq p(x)$ for every $x\in X$ and the space $(X,p)$ is complete. According to Banach Theorem on homomorphism, the identity mapping $id:(X,\|\cdot\|)\to (X,p)$ is continuous. Therefore for every $n\in\mathbb N$ there exists $\delta>0$ such that the inequality $\|x\|<\delta$ implies $p(x)\leq \frac{1}{2n}$. Then $\|f_i(x)e_i\|\leq \frac{1}{n}$. Hence, $f_i$ is continuous and $|f_i(x)|\leq \frac{1}{n}$ if $\|e_i\|\geq 1$. Moreover, the family $(T_{i,j}:i<j\leq i_0)$ of operators $T_{i,j}:X\to X$, $T_{i,j}(x)=\sum_{i\leq \alpha < j}f_\alpha(x)e_\alpha$, is equicontinuous too.

Now we prove the second part. Let $j_1$ be the maximal element of the set $J_1$ of all limited elements in $I_0$ ($j_1$ exists being equal the minimal element of the set $\{i\in I_0: (\forall j\in J_1)\,(j\leq i)\}$) containing $i_0$. It follows from \cite[p.~86]{Sh1} that there exists $i_1<j_1$ such that $\|T_{i_1,j_1}x\|\leq \frac{\varepsilon}{2}$ for every $x\in K$. Let $j_2$ be the maximal element of the set $J_2$ of all limit elements in $I_1=\{i\in I_0:i\leq i_1\}$. Note that $j_2<j_1$ and the set $\{i\in I_0:j_2\leq i< i_1\}$ is finite. Using \cite[p.~86]{Sh1} we choose $i_2<j_2$ such that $\|T_{i_2,j_2}x\|\leq \frac{\varepsilon}{2^2}$ for every $x\in K$. And so on. Since the set $I_0$ does not contain a strictly decreasing sequence, there exists an $n\in\mathbb N$ such that the set $J_{n+1}$ of all limit elements in $I_n=\{i\in I_0:i< i_n\}$ is empty, that is $I_n$ is finite. It remains to put $J=I_n\cup\left(\cup_{1\leq k\leq n-1}\{i\in I_0:j_{k+1}\leq i< i_k\}\right)$.
$\Box$\end{proof}

\begin{theorem} \label{th:3}
Let $X$ be a K\"{o}the $F$-space on $(\Omega, \Sigma, \mu)$, $Z$ be the set of all signs in $X$, $Y$ be an $F$-space with a transfinite basis, $T_1\in\mathcal L(X,Y)$ be a narrow operator, $T_2\in\mathcal L(X,Y)$ be a narrow operator such that the set $K=T_2(Z)$ is a relatively compact set in $Y$. Then the operator $T=T_1+T_2$ is narrow.
\end{theorem}

\begin{proof} The proof follows immediately from Proposition \ref{pr:4} and Corollary \ref{cor:2}.
$\Box$\end{proof}

\begin{theorem} \label{th:2}
Let $X$ be a K\"{o}the $F$-space on $(\Omega, \Sigma, \mu)$, $Z$ be the set of all signs in $X$, $Y$ be a locally convex $F$-space, $T_1\in\mathcal L(X,Y)$ be a narrow operator, $T_2\in\mathcal L(X,Y)$ be a narrow operator such that the set $K=T_2(Z)$ is relatively compact in $Y$. Then the operator $T=T_1+T_2$ is narrow.
\end{theorem}

\begin{proof}
It is sufficient to prove that for every $\varepsilon>0$ there exists a mean zero sign $x$ on $\Omega$ such that $\|T x\|\leq \varepsilon$.

Fix $\varepsilon>0$ and show that there is a mean zero sign $x$ on $\Omega$ such that $\|T_1 x\|\leq \frac{\varepsilon}{2}$ and $\|T_2 x\|\leq \frac{\varepsilon}{2}$.

Choose an absolutely convex open 0-neighborhood $V$ in $Y$ such that $\|y\|\leq \frac{\varepsilon}{5}$ for every $y\in V$. If $K_1=K\setminus \{y\in Y:\|y\|\leq \frac{\varepsilon}{2}\}=\emptyset$ then using the narrowness of $T_1$ it is sufficient to choose a mean zero sign $x$ on $\Omega$ such that $\|T_1 x\|\leq \frac{\varepsilon}{2}$.

Now let $K_1\ne\emptyset$. Choose $y_1,\dots ,y_n\in K_1$ such that $K_1\subseteq \bigcup\limits_{k=1}^{n}(y_k+V)$. Note that $(y_k+V)\cap V=\emptyset$ for every $k\leq n$. Then for every $k\leq n$ there exists a linear continuous functional $f_k:Y\to\mathbb R$ which strictly separates the sets $y_k+V$ and $V$ (see, for example, \cite[p.64]{Sh}). Without loss of generality we may and do assume that $f_k(y_k)=1$, moreover, $f_k(y)<\frac{1}{2}$ for every $y\in V$. Consider the continuous linear operators $S:Y\to l_\infty^n$, $Sy=(f_k(y))_{k=1}^n$ and $S_1:X\to l_\infty^n$, $S_1(x)=S(T_2x)$. Since the operator $T_2$ is narrow, the operator $S_1$ is narrow too. Therefore by Theorem \ref{th:1}, there exists a mean zero sign $x$ on $\Omega$ such that $\|T_1x\|\leq \frac{\varepsilon}{2}$ and $\|S_1x\|\leq \frac{1}{2}$.

Show that $\|T_2x\|\leq \frac{\varepsilon}{2}$. Assume that $\|T_2x\|> \frac{\varepsilon}{2}$. Then $y_0=T_2x\in K_1$ and there exists an $k\leq n$ such that $y_0\in (y_k+V)$. Therefore $y_k-y_0\in V$, $f_k(y_k-y_0)<\frac{1}{2}$ and $f_k(y_0)>f_k(y_k)-\frac{1}{2}=\frac{1}{2}$. Thus $\|S_1x\|\geq f_k(T_2(x))=f_k(y_0)>\frac{1}{2}$, which implies a contradiction.
$\Box$\end{proof}

\section{Examples, a question}

The conclusion of Proposition \ref{pr:3} is valid for every linear continuous operator $T\in\mathcal L(X,Y)$ with $\lim\limits_{\mu(A)\to 0}\|T\mathbf{1}_A\|=0$, in particular, if $X$ has an absolute continuous norm of the unit. The next example shows that the conclusion of Proposition \ref{pr:3} cannot be replaced with the following stronger conclusion:
{\it for every $\varepsilon>0$ there is $\delta>0$ such that $\|T\mathbf{1}_A\|\leq \varepsilon$ if $\mu(A)\leq \delta$}. The example construction is similar to the construction used in \cite[Example 4.2]{MMP}.

\begin{example} \label{prob:3}
There exists a linear continuous strictly narrow functional $f$ on $L_\infty$ such that $f(\mathbf{1}_{A_n})=1$ for some sequence of measurable sets $A_n$ with $\lim\limits_{n\to\infty}\mu(A_n)=0$.
\end{example}

\begin{proof} We make a construction for the space $L_\infty([0,1]^2)$. Denote by $\mu_1$ and $\mu_2$ the Lebesgue measure on $[0,1]$ and $[0,1]^2$ respectively, and by $\Sigma_1$ and $\Sigma_2$ the system of all measurable subsets of $[0,1]$ and $[0,1]^2$ respectively. Let $P:L_\infty([0,1]^2)\to L_\infty([0,1])$ be the conditional expectation operator which is defined by $Px(t)=\int\limits_{[0,1]}x(t,s)d\mu_1(s)$. It follows from \cite[Theorem 4.10]{PR} that $P$ is strictly narrow. Choose a maximal system ${\mathcal A}\subseteq \Sigma_1$ such that \mbox{$\mu_1(\bigcap\limits_{k=1}^n B_k)>0$} for every $n\in\mathbb N$ and $A_1,\dots , A_n\in {\mathcal A}$. The maximality of ${\mathcal A}$ implies that either $A\in {\mathcal A}$, or $B\in {\mathcal A}$ for every $A,B\in\Sigma_1$ with $[0,1]=A\bigsqcup B$. Moreover, all the sets $A\in\Sigma_1$ with $\mu_1(A)=1$ belong to the system ${\mathcal A}$. Consider the finitely additive measure $\lambda:\Sigma \to \{0,1\}$ on $[0,1]$ defined by
$$\lambda(A)=\left\{\begin{array}{ll}
                         1, & A\in {\mathcal A}\\
                         0, & A\not\in {\mathcal A}.
                       \end{array}
 \right.$$ Observe that the formula $g(x)=\int x d\lambda$ defines a linear continuous functional $g$ on $L_\infty([0,1])$ and the formula $f(x)=g(Px)$ defines a linear continuous strictly narrow functional $f$ on $L_\infty([0,1]^2)$.

Choose a sequence of sets $B_n\in {\mathcal A}$ such that $\mu_1(B_n)\leq \frac{1}{n}$ and set $A_n=B_n\times [0,1]$ for every $n\in\mathbb N$. It is clear that
$\lim\limits_{n\to\infty}\mu_2(A_n)=0$ and $f(\mathbf{1}_{A_n})=g(\mathbf{1}_{B_n})=1$ for every $n\in\mathbb N$.
$\Box$\end{proof}

The following example shows that the assumption of the relative compactness of $T(Z)$ in Propositions \ref{pr:3} and \ref{pr:1} is weaker than the compactness of $T$.

\begin{example} \label{prob:4}
There exists a linear continuous strictly narrow non-compact operator $T:L_1\to l_1$ such that the set $T(Z)$ is compact, where $Z$ is the set of all signs in $L_1$.
\end{example}

\begin{proof} For every $n\in\mathbb N$ set $A_n=[\frac{1}{2^n}, \frac{1}{2^{n-1}}]$ and for every $x\in L_1$ set $Tx=\left(\int_{A_n}xd\mu\right)_{n=1}^\infty$. It is clear that the operator $T:L_1\to l_1$ is linear and continuous with $\|T\|=1$. Moreover, for every $A\in \Sigma^+$ there is a mean zero sign $A$ such that $\int_{A_n}xd\mu=0$ for every $n\in\mathbb N$. Thus, $T$ is strictly narrow. Since $|\int_{A_n}zd\mu|\leq \frac{1}{2^n}$ for every $n\in \mathbb N$ and $z\in Z$, the set $T(Z)$ is relatively compact. Finally, $T$ is non-compact, because the standard basis of $l_1$ is contained in the image of the unit ball of $L_1$.
$\Box$\end{proof}

The following question naturally arises in connection with Theorems \ref{th:2} and \ref{th:3}.

\begin{problem} \label{prob:5}
Let $0<p<1$. Is a sum of two narrow operators from $\mathcal L(L_\infty, L_p)$, at least one of which is compact, narrow?
\end{problem}


\end{document}